\documentclass[14pt,a4paper]{article}

\usepackage{mathrsfs}
\usepackage{amscd}
\usepackage{hyperref}
\hypersetup{colorlinks=true,citecolor=green,linkcolor=blue,urlcolor=blue}
\usepackage{tipa}
\usepackage{amssymb}
\usepackage{amsfonts}
\usepackage{stmaryrd}
\usepackage{amssymb}
\usepackage{titling}
\usepackage{CJK}  
\usepackage{indentfirst} 

\usepackage{latexsym}   
\usepackage{bm}         

\usepackage{pifont}
\usepackage{bm}
\usepackage{amsmath,amssymb,amsfonts}
\usepackage{indentfirst}

\usepackage{xcolor}
\usepackage{cases}
\usepackage{wasysym}
\usepackage{amsthm}

\newtheorem{theorem}{Theorem}[section]

\newtheorem{remark}[theorem]{Remark}

\CJKtilde   

\begin{document}

\title{\large \textbf{Mod-$p$ maximal compact inductions do not have irreducible admissible subrepresentations}}
\date{}
\author{\textbf{Peng Xu}}
\maketitle

\begin{abstract}
Let $p$ be a prime number. We show in this short note that mod-$p$ maximal compact inductions of a $p$-adic split reductive group do not have irreducible admissible subrepresentations.
\end{abstract}

\section{Introduction}

Let $F$ be a non-archimedean local field of residue characteristic $p$, and $G$ be a $p$-adic split reductive group defined over $F$. Let $K$ be a hyperspecial maximal compact subgroup of $G$. Let $\sigma$ be an irreducible smooth $\overline{\mathbf{F}}_p$-representation of $K$. We show in this short note:

\begin{theorem}\label{main in intro}
The compactly induced representation $\textnormal{ind}^G _K \sigma$ does not have irreducible admissible subrepresentations.
\end{theorem}

\section{Proof of Theorem \ref{main in intro}}

Let $\sigma$ be an irreducible smooth $\overline{\mathbf{F}}_p$-representation of $K$.

\begin{proof}[Proof of Theorem \ref{main in intro}]

Assume $\pi$ is an irreducible admissible $\overline{\mathbf{F}}_p$-representation of $G$ contained in $\textnormal{ind}^G _K \sigma$, i.e., we are given a $G$-embedding
\begin{center}
$\iota: \pi \hookrightarrow \textnormal{ind}^G _K \sigma$.
\end{center}
Take an irreducible smooth $\overline{\mathbf{F}}_p$-representation $\sigma_1$ of $K$ contained in $\pi$. By Frobenius reciprocity, we get a non-zero $G$-map $\theta$ in the space $\text{Hom}_G (\textnormal{ind}^G _K \sigma_1, \pi)$, which will be denoted by $H(\sigma_1, \pi)$. Since $\pi$ is admissible, the space $H(\sigma_1, \pi)$ is finite dimensional. By composition, the space $H(\sigma_1, \pi)$ is a right module over the spherical Hecke algebra $\mathcal{H} (K, \sigma_1): =\text{End}_G (\textnormal{ind}^G _K \sigma_1)$. As $G$ is split, the algebra $\mathcal{H} (K, \sigma)$ is commutative (\cite[Corollary 1.3]{Her2011a}). Therefore we may replace $\theta$ by an eigenvector. That is to say, there is a character $\chi: \mathcal{H} (K, \sigma_1)\rightarrow \overline{\mathbf{F}}_p$ so that $\theta$ is annihilated by the kernel of $\chi$.

\smallskip

Now under our assumption, the composition $\iota\circ \theta$ is a non-zero map in $\text{Hom}_G (\textnormal{ind}^G _K \sigma_1, \textnormal{ind}^G _K \sigma)$. We take a non-zero $T$ in the kernel of $\chi$. We get
\begin{center}
$(\iota\circ \theta)\circ T =\iota \circ (\theta\circ T)=0$.
\end{center}
As $T$ and $\iota\circ \theta$ are both non-zero, we get a contradiction by the argument of \cite[Corollary 6.5]{Her2011b}.
\end{proof}

\begin{remark}
Note that we have assumed $\pi$ is admissible in the theorem. However, in certain cases such an assumption is not necessary, say $G= GL_2/ U(2, 1)$, as in both cases a weaker substitute, i.e., the existence of Hecke eigenvalues (\cite{B-L94}, \cite{X2018a}), is available. Note also that Daniel Le proved recently that there are non-admissible irreducible mod-$p$ smooth representations of $GL_2 (\mathbf{Q}_{p^3})$ (\cite{Le2018}).
\end{remark}

\begin{remark}
One interest of Theorem \ref{main in intro} is to compare it with the \textbf{complex} case: when $G$ is a classical group, a complex maximal compactly induced representation of $G$ might often happen to be irreducible.
\end{remark}

\bibliographystyle{amsalpha}
\bibliography{new}

\end{document}